\documentclass{article}

\usepackage{amsmath, amssymb, amsthm, amsfonts, bm}
\usepackage{enumerate}

\numberwithin{equation}{section}
\newtheorem{theorem}{Theorem}[section]
\newtheorem{corollary}{Corollary}[section]
\newtheorem{lemma}{Lemma}[section]
\theoremstyle{definition}
\newtheorem{definition}{Definition}[section]
\theoremstyle{remark}

\newcommand{\Rnum}[1]{\uppercase\expandafter{\romannumeral #1\relax}}
\newcommand{\rnum}[1]{\romannumeral #1\relax}
\newcommand{\mr}[1]{\mathrm{#1}}
\newcommand{\mb}[1]{\mathbb{#1}}
\newcommand{\mc}[1]{\mathcal{#1}}

\DeclareMathOperator{\Av}{Av} % Average operator

\def\clap#1{\hbox to 0pt{\hss#1\hss}}

\title{The Bellman function of the dyadic maximal operator in connection with the Dyadic Carleson Imbedding Theorem}
\author{Eleftherios N. Nikolidakis}

\begin{document}
\maketitle

\begin{abstract}
We provide an alternative proof and expression of the Bellman function of the dyadic maximal operator in connection with the Dyadic Carleson Imbedding Theorem, which appears in \cite{10}. We also evaluate the Bellman function of four variables of the dyadic maximal operator
\end{abstract}

\section{Introduction} \label{sec:1}
The dyadic maximal operator on $\mb R^n$ is a useful tool in analysis and is defined by
\begin{equation} \label{eq:1p1}
\mc M_d\phi(x) = \sup\left\{ \frac{1}{|S|} \int_S |\phi(u)|\,\mr du: x\in S,\ S\subseteq \mb R^n\ \text{is a dyadic cube} \right\},
\end{equation}
for every $\phi\in L^1_\text{loc}(\mb R^n)$, where $|\cdot|$ denotes the Lebesgue measure on $\mb R^n$, and the dyadic cubes are those formed by the grids $2^{-N}\mb Z^n$, for $N=0, 1, 2, \ldots$.\\
It is well known that it satisfies the following weak type (1,1) inequality
\begin{equation} \label{eq:1p2}
\left|\left\{ x\in\mb R^n: \mc M_d\phi(x) > \lambda \right\}\right| \leq \frac{1}{\lambda} \int_{\left\{\mc M_d\phi > \lambda\right\}} |\phi(u)|\,\mr du,
\end{equation}
for every $\phi\in L^1(\mb R^n)$, and every $\lambda>0$,
from which it is easy to get the following  $L^p$-inequality
\begin{equation} \label{eq:1p3}
\|\mc M_d\phi\|_p \leq \frac{p}{p-1} \|\phi\|_p,
\end{equation}
for every $p>1$, and every $\phi\in L^p(\mb R^n)$.
It is easy to see that the weak type inequality \eqref{eq:1p2} is the best possible. 

It has also been proved that \eqref{eq:1p3} is best possible (see \cite{2} and \cite{3} for general martingales and \cite{15} for dyadic ones).
An approach for studying the behaviour of this maximal operator in more depth is the introduction of the so-called Bellman functions which play the role of generalized norms of $\mc M_d$. Such functions related to the $L^p$-inequality \eqref{eq:1p3} have been precisely identified in \cite{8}, \cite{10} and \cite{14}. For the description of the Bellman functions of $\mc M_d$, we use the notation $\Av_E(\psi)=\frac{1}{|E|} \int_E \psi$, whenever $E$ is a Lebesgue measurable subset of $\mb R^n$ of positive measure and $\psi$ is a real valued measurable function defined on $E$. We fix a dyadic cube  $Q$ and define the localized maximal operator $\mc M'_d\phi$ as in \eqref{eq:1p1} but with the dyadic cubes $S$ being assumed to be contained in $Q$. Then for every $p>1$ we let
\begin{equation} \label{eq:1p4}
B_p(f,F)=\sup\left\{ \frac{1}{|Q|} \int_Q (\mc M'_d\phi)^p: \Av_Q(\phi)=f,\ \Av_Q(\phi^p)=F \right\},
\end{equation}
where $\phi$ is nonnegative in $L^p(Q)$ and the variables $f, F$ satisfying $0<f^p\leq F$. By a scaling argument it is easy to see that \eqref{eq:1p4} is independent of the choice of $Q$ (so we may choose
$Q$ to be the unit cube $[0,1]^n$).
In \cite{10}, the function \eqref{eq:1p4} has been precisely identified for the first time. The proof has been given in a much more general setting of tree-like structures on probability spaces.

More precisely we consider a non-atomic probability space $(X,\mu)$ and let $\mc T$ be a family of measurable subsets of $X$, that has a tree-like structure similar to the one in the dyadic case (the exact definition will be given in Section \ref{sec:2}).
Then we define the dyadic maximal operator associated to $\mc T$, by
\begin{equation} \label{eq:1p5}
\mc M_{\mc T}\phi(x) = \sup \left\{ \frac{1}{\mu(I)} \int_I |\phi|\,\mr \; d\mu: x\in I\in \mc T \right\},
\end{equation}
for every $\phi\in L^1(X,\mu)$, $x\in X$.

This operator is related to the theory of martingales and satisfies essentially the same inequalities as $\mc M_d$ does. Now we define the corresponding Bellman function of four variables of $\mc M_{\mc T}$, by
\begin{multline} \label{eq:1p6}
B_p^{\mc T}(f,F,L,k) = \sup \left\{ \int_K \left[ \max(\mc M_{\mc T}\phi, L)\right]^p\mr \; d\mu: \phi\geq 0, \int_X\phi\,\mr \; d\mu=f, \right. \\  \left. \int_X\phi^p\,\mr \; d\mu = F,\ K\subseteq X\ \text{measurable with}\ \mu(K)=k\right\},
\end{multline}
the variables $f, F, L, k$ satisfying $0<f^p\leq F $, $L\geq f$, $k\in (0,1]$.
The exact evaluation of \eqref{eq:1p6} is given in \cite{10}, for the cases where $k=1$ or $L=f$. In the first case the author (in \cite{10}) precisely identifies the function $B_p^{\mc T}(f,F,L,1)$ by evaluating it in a first stage for the case where $L=f$. That is he precisely identifies $B_p^{\mc T}(f,F,f,1)$ (in fact $B_p^{\mc T}(f,F,f,1)=F \omega_p (\frac{f^p}{F})^p$, where                         $\omega_p: [0,1] \to [1,\frac{p}{p-1}]$ is the inverse function $H^{-1}_p$, of $H_p(z) = -(p-1)z^p + pz^{p-1}$). Then using several calculus argument he provides the evaluation of $B_p^{\mc T}(f,F,L,1)$ for every $L\geq f$. Now in \cite{14} the authors give a direct proof of the evaluation of $B_p^{\mc T}(f,F,L,1)$ by using alternative methods. In fact they prove a sharp symmetrization principle that holds for the dyadic maximal operator, which is stated as Theorem 2.1 (see Section \ref{sec:2}).

In the second case, where $L=f$, the author (in \cite{10}) uses the evaluation of $B_p^{\mc T}(f,F,f,1)$ and provides a proof of the more general $B_p^{\mc T}(f,F,f,k)$, $k\in (0,1]$. We write from now on this function as $B_p^{\mc T}(f,F,k)$. This function is related to the Dyadic Carleson Imbedding Theorem and in fact, as is proved in \cite{10}, the following is true
\begin{multline} \label{eq:1p7}
B_p^{\mc T}(f,F,k) = \sup\left\{ \sum_{I\in \mc T} \lambda_I(Av_I(\phi))^p,\phi \geq 0,\ \int_X \phi\,\mr \; d\mu = f,\ \int_X \phi^p\,\mr \; d\mu = F,  \right. \\  \left. \text{and the nonegative }\ \lambda_I \text{'s satisfy} \sum_{J\in \mathcal{T} : J \subseteq I}\lambda_J \leq \mu (I) \text{ for every } I\in \mathcal{T}  \right. \\  \left. \text{ and } \sum_{I\in \mathcal{T}}\lambda_I=k \right\}.
\end{multline}

As an immediate step for the evaluation of $B_p^{\mc T}(f,F,k)$ in \cite{10}, it is provided an alternative expression for this function. This is stated in the following theorem

\begin{theorem}\label{thm:1p1}
The following is true
\begin{multline}\label{eq:1p8}
B_p^{\mc T}(f,F,k)= \sup \left\lbrace \left( F-\frac{(f-B)^p}{(1-k)^{p-1}} \right) \omega_p \left( \frac{B^p}{k^{p-1}\left( F-\frac{(f-B)^p}{(1-k)^{p-1}} \right)}\right) \ \ :\right. \\  \left. \text{for all } B\in [0,f] \text{ such that }  h_k(B) \leq F \right\rbrace,
\end{multline}
where $h_k$ is defined by $h_k(B)=\frac{(f-B)^p}{(1-k)^{p-1}}+\frac{B^p}{k^{p-1}}$.
\end{theorem}

After proving the above theorem, the author in \cite{10}, precisely evaluated $B_p^{\mc T}(f,F,k)$, by using a chain of calculus arguments.
In Section \ref{sec:3} we provide an alternative proof of Theorem 1.1. Now in view of the symmetrization principle that appears in \cite{14} (see Theorem 2.1 below) we conclude that

\begin{multline}\label{eq:1p9}
B_p^{\mc T}(f,F,k)= \sup \left\lbrace \int_0^k \left( \frac{1}{t} \int_0^t g\right)^p dt \ \ : \text{ where } g:(0,1]\longrightarrow \mathbb{R}^{+} \text{ is} \right. \\  \left. \text{ non-increasing, } \int_0^1 g=f,\; \int_0^1 g^p =F \right\rbrace.
\end{multline}

In Section \ref{sec:4} we prove the following

\begin{theorem}
There exists a function $g=g_k:(0,1]\longrightarrow \mathbb{R}^{+}$ non-increasing and continuous, satisfying $\int_0^1 g=f$ and $\int_0^1 g^p =F$, for which the supremum in \eqref{eq:1p9} is attained.
\end{theorem}
Moreover we explicitly construct the function $g_k$, mentioned above.
At last in Section 5 we give a proof of the evaluation of the Bellman function of four variables (1.6) for the dyadic maximal operator in terms of $B_p^{\mc T}(f,F,k)$.
 
We note also that further study of the dyadic maximal operator can be seen in \cite{14} where a sharp symmetrization principle for this operator is presented. There are several problems in Harmonic Analysis where Bellman functions naturally arise. Such problems (including the dyadic Carleson Imbedding Theorem and weighted inequalities) are described in \cite{12,13} and others. We should mention also that the exact computation of a Bellman function is a difficult task which is connected with the deeper structure of the corresponding Harmonic Analysis problem. Thus far several Bellman functions have been computed (see \cite{2,9,11}). For more recent developments we refer to \cite{1,6,7}, while for the study of general theory of maximal operators one can consult \cite{4,5}.

In this paper, as in our previous ones we use Bellman functions as a mean to get in deeper understanding of the corresponding maximal operators and we are not using the standard techniques as Bellman dynamics and induction, corresponding PDE's, obstacle conditions etc. Instead our methods being different from the Bellman function technique, we rely on the combinational structure of these operators. For such approaches, which enables us to study and solve problems as the one which is described in this article one can see \cite{8,9,10,11,14}.\\

\section{Preliminaries} \label{sec:2}
Let $(X,\mu)$ be a nonatomic probability space. We give the following

\begin{definition} \label{def:2p1}
A set $\mc T$ of measurable subsets of $X$ will be called a tree if the following conditions are satisfied:
\begin{enumerate}[i)]
\item $X\in\mc T$ and for every $I\in\mc T$ we have that $\mu(I) > 0$.
\item For every $I\in\mc T$ there corresponds a finite or countable subset $C(I) \subseteq \mc T$ containing at least two elements such that
\vspace{-5pt}
\begin{enumerate}[a)]
\item the elements of $C(I)$ are pairwise disjoint subsets of $I$.
\item $I = \cup\, C(I)$.
\end{enumerate}
\item $\mc T = \cup_{\substack{ m\geq 0}} \mc T_{(m)}$, where $\mc T_{(0)} = \left\{ X \right\}$ and $\mc T_{(m+1)} = \cup_{I\in \mc T_{(m)}} C(I)$.
\item We have $\lim_{m\to\infty} \left( \sup_{I\in \mc T_{(m)}} \mu(I) \right)= 0 $
\item The tree $\mc T$ differentiates $L^1(X,\mu)$. That is for every $\phi \in L^1(X,\mu)$ it is true that \[ \lim_{\substack{x\in I \in \mc T \\ \mu (I)\to 0}} \frac{1}{\mu (I)}\int_I \phi \; d\mu = \phi(x),\] for $\mu-$almost every $x\in X$.
\end{enumerate}
\end{definition}

Then we define the dyadic maximal operator corresponding to $\mc T$ by
\begin{equation}\label{eq:2p1}
M_{\mc T}\phi (x)=sup \left\lbrace \frac{1}{\mu(I)} \int_I \mid \phi \mid \; d\mu \; : x\in I\in \mc T  \right\rbrace,
\end{equation}
for every $\phi \in L^1(X,\mu)$, $x\in X$.

We give the following which appears in \cite{10}.

\begin{lemma} \label{lem:2p1}
For every $I\in \mc T$ and every $\alpha$ such that $0<a<1$ there exists a subfamily $\mc F(I) \subseteq \mc T$ consisting of pairwise disjoint subsets of $I$ such that
\[
\mu\!\left( \underset{J\in\mc F(I)}{\bigcup} J \right) = \sum_{J\in\mc F(I)} \mu(J) = (1-\alpha)\mu(I).
\]
\end{lemma}

\begin{definition} \label{def:2p2}
Let $\phi: (X,\mu)\longrightarrow \mathbb{R}^{+}$. Then $\phi^{*}:(0,1]\longrightarrow \mathbb{R}^{+}$ is defined as the unique non-increasing, left continuous and equimeasurable to $\phi$ function on $(0,1]$.
\end{definition}

There are several formulas that express $\phi^{*}$, in terms of $\phi$.
One of them is as follows:
\[
\phi^{*}(t)=\inf \left( \left\lbrace y>0: \mu \left( \left\lbrace x\in X: \phi(x)>y\right\rbrace\right)<t \right\rbrace \right),
\]
for every $t\in (0,1]$. An equivalent formulation of the non increasing rearrangement can be given by
\[
\phi^{*}(t)=\sup_{e\subseteq X, \\  \mu(e)\geq t} \left[ \inf_{x\in e} \phi (x) \right],
\]
for any $t\in (0,1]$.

In \cite{20} one can see the following symmetrization principle for the dyadic maximal operator $M_{\mc T}$.

\begin{theorem}\label{thm:2p1}
Let $g:(0,1]\longrightarrow \mathbb{R}^{+}$ be non-increasing and $G_1,G_2$ be non-decreasing and non-negative functions defined on $[0,+\infty)$. Then the following is true, for any $k\in (0,1]$
\begin{multline*}
\sup \left\lbrace \int_K G_1(M_{\mc T}\phi)\,G_2(\phi)\; d\mu: \phi^{*}=g \text{ and } \mu(K)=k  \right\rbrace=\\
= \int_0^k G_1\left( \frac{1}{t}\int_0^t g \right) G_2\left(g(t)\right) \; dt.
\end{multline*}
\end{theorem}

We also state the following, which is a standard fact in the theory of real functions.

\begin{lemma} \label{lem:2p2}
Let $g_1, g_2:(0,1]\longrightarrow \mathbb{R}^{+}$ be non-increasing functions, such that
\[
\int_0^1 G\left( g_1(t) \right)\; dt\leq \int_0^1 G\left( g_2(t) \right)\; dt
\] for every $G:[0,+\infty)\longrightarrow [0,+\infty)$ non-decreasing. Then the inequality $g_1(t) \leq g_2(t)$ holds almost everywhere on $(0,1]$
\end{lemma}

Fix now $k\in(0,1)$, $p>1$ and consider the function
\begin{equation} \label{eq:2p2}
h_k(B) = \frac{(f-B)^p}{(1-k)^{p-1}} + \frac{B^p}{k^{p-1}},
\end{equation}
defined for $B\in [0,f]$.

We also define
\begin{equation} \label{eq:2p3}
\mathcal R_k(B) = \left(F-\frac{(f-B)^p}{(1-k)^{p-1}}\right)\omega_p\!\left(\frac{B^p}{k^{p-1}\left(F-\frac{(f-B)^p}{(1-k)^{p-1}}\right)}\right)^p,
\end{equation}
for $B$ such that $B\in [0,f]$ and $h_k(B)\leq F$.
Note that $\mathcal R_k(B)$ is defined for all $B\in [0,f]$ for which $h_k(B) \leq F$ or equivalently:
\[
\frac{(f-B)^p}{(1-k)^{p-1}} + \frac{B^p}{k^{p-1}} \leq F \iff 0\leq \frac{B^p}{k^{p-1}\left[F-\frac{(f-B)^p}{(1-k)^{p-1}}\right]} \leq 1
\]
so that \eqref{eq:2p3} makes sense in view of the definition of $\omega_p$.

Then as is mentioned in \cite{10} the domain of definition of $\mathcal R_k$ is an interval $[p_0(f,F,k),\, p_1(f,F,k)]\subseteq [0,f]$. Moreover the following technical lemma is proved in \cite{10}. It describes the properties of the unique interior point $B_0$ of $[p_0(f,F,k),\, p_1(f,F,k)]$ at which the function $\mathcal R_k$ attains its maximum in its domain of definition.

\begin{lemma} \label{lem:2p3}
The function $\mathcal R_k$ defined on $\left[p_0(f,F,k),\, p_1(f,F,k)\right]$ attains its absolute maximum at a unique interior point $\displaystyle B_0\in\left(kf, \min\left(\frac{pk}{p-1+k},p_1(f,F,k)\right)\right)$.
Moreover $B_0$ satisfies
\begin{equation*}
H_p\left(\frac{B_0}{k}\frac{1-k}{f-B_0}\right)=\frac{B_0^p}{k^{p-1}\left[F-\frac{(f-B_0)^p}{(1-k)^{p-1}}\right]}
\end{equation*}

\end{lemma}

\section{The Bellman function $B_p(f,F,k)$}  \label{sec:3}

\begin{lemma} \label{lem:3p1}
For any $\phi : (X,\mu) \to \mb R^+$ integrable, the following inequality is true
\[
(M_{\mc T}\phi)^\star (t) \leq \frac{1}{t} \int_0^t \phi^\star(u)\,\mr du,\ \text{for every}\ t\in (0,1].
\]
\end{lemma}

\begin{proof}
By Theorem \ref{thm:2p1} we have for any $G:[0,+\infty)\longrightarrow [0,+\infty)$ non-decreasing
\begin{equation} \label{eq:3p1}
\int_X G(M_{\mc T}\phi)\,\mr d\mu \leq \int_0^1 G\!\left(\frac{1}{t}\int_0^t \phi^\star(u)\,\mr du\right)\mr dt.
\end{equation}
Since $G$ is non-decreasing we have that
\[
\left[G(M_{\mc T}\phi)\right]^\star(t) = G\left[(M_{\mc T}\phi)^\star\right](t),\ \text{for almost every}\ t\in(0,1].
\]
Thus $\int_0^1 G\left[(M_{\mc T}\phi)^\star\right](t)\, dt = \int_0^1\left[G(M_{\mc T}\phi)\right]^\star(t)\,\mr dt =
\int_X G(M_{\mc T}\phi)\, d\mu \leq \\  \int_0^1 G\!\left(\frac{1}{t}\int_0^t\phi^\star(u)\, du\right) dt$, by \eqref{eq:3p1}.
Thus by Lemma \ref{lem:2p2} we immediately conclude that
\begin{equation} \label{eq:3p2}
(M_{\mc T}\phi)^\star(t) \leq \frac{1}{t} \int_0^t \phi^\star(u)\,\mr du,
\end{equation}
almost everywhere on $(0,1]$. Since now $(M_{\mc T}\phi)^\star$ is left continuous, we conclude that \eqref{eq:3p2} should hold everywhere on $(0,1]$, and in this way we derive the proof of our Lemma.
\end{proof}

We also present a simpler proof of Lemma \ref{lem:3p1}.

\begin{proof}[2nd proof of Lemma \ref{lem:3p1}] ~ \\
Suppose that we are given $\phi : (X,\mu)\to \mb R^+$ integrable and $t\in (0,1]$ fixed. We set $A = \frac{1}{t}\int_0^t \phi^\star(u)\,\mr du$. Then obviously $A \geq \int_0^1\phi^\star(u)\,\mr du = f$, by the fact that $\phi^\star$ is non-increasing on $(0,1]$. We consider the set $E = \{M_{\mc T}\phi > A\} \subseteq X$. \\
Then by the weak type inequality \eqref{eq:1p2} for $M_{\mc T} \phi$, we have that
\begin{multline} \label{eq:3p3}
\mu(E) < \frac{1}{A} \int_E |\phi|\,\mr d\mu \Rightarrow  \\
A = \frac{1}{t} \int_0^t \phi^\star(u)\,\mr d\mu < \frac{1}{\mu(E)} \int_E \phi\,\mr d\mu \leq \frac{1}{\mu(E)} \int_0^{\mu(E)} \phi^\star(u)\,\mr du,
\end{multline}
where the last inequality in \eqref{eq:3p3} holds due to the definition of $\phi^\star$. Since $\phi^\star$ is non-increasing we must have from \eqref{eq:3p3}, that $\mu(E)<t$. But $\mu(E) = \left|\left\{(M_{\mc T}\phi)^\star(t) > A \right\}\right|$ since $(M_{\mc T}\phi)$ and $(M_{\mc T}\phi)^\star$ are equimeasurable.
Since $(M_{\mc T}\phi)^\star$ is non-increasing and because of the fact that $\mu(E) < t$ we conclude that $\{(M_{\mc T}\phi)^\star > A\} = (0,\gamma)$ for some $\gamma<t$. Thus $t\notin \{(M_{\mc T}\phi)^\star > A \} \Rightarrow (M_{\mc T}\phi)^\star(t) \leq A = \frac{1}{t}\int_0^t\phi^\star(u)\,\mr du$,
which is the desired result.
\end{proof}
We are now in position to state and prove

\begin{lemma} \label{lem:3p2}
Let $\phi : (X,\mu) \to \mb R^+$ be such that $\int_X \phi \,\mr d\mu = f$ and $\int_X \phi^p\,\mr d\mu = F$ where $0<f^p\leq F$. Suppose also that we are given a measurable subset $K$ of $X$ such that $\mu(K)=k$, where $k$ is fixed such that $k\in (0,1]$. Then the following inequality is true:
\[
\int_K (M_{\mc T}\phi)^p\,\mr d\mu \leq \int_0^k [\phi^\star(u)]^p\,\mr du \cdot \omega_p\!\!\left(\frac{\left(\int_0^k\phi^\star(u)\,\mr du\right)^p}{k^{p-1}\int_0^k[\phi^\star(u)]^p\,\mr du}\right)^p.
\]
\end{lemma}

\begin{proof}
We obviously have that
\begin{equation} \label{eq:3p4}
\int_K(M_{\mc T} \phi)^p\, d\mu \leq \int_0^k [(M_{\mc T} \phi)^\star]^p(t)\, dt.
\end{equation}
We evaluate the right-hand side of \eqref{eq:3p4}. We have:
\begin{equation}  \label{eq:3p5}
\int_0^k \left[ (M_{\mc T}\phi)^\star)\right]^p dt \leq \int_0^k \left(\frac{1}{t} \int_0^t \phi^\star(u)\,\mr du\right)^p\mr dt
\end{equation}
by using Lemma \ref{cor:3p1}. Additionally
\begin{multline}  \label{eq:3p6}
\int_0^k \left(\frac{1}{t} \int_0^t \phi^\star(u)\,\mr du\right)^p\mr dt =
\int_{\lambda=0}^{+\infty} p\lambda^{p-1} \left|\left\{t\in(0,k] : \frac{1}{t}\int_0^t\phi^\star \geq \lambda\right\}\right|\mr d\lambda = \\
\int_{\lambda=0}^{f_k} + \int_{\lambda=f_k}^{+\infty} p\lambda^{p-1}\left|\left\{t\in (0,k] : \frac{1}{t}\int_0^t\phi^\star\geq \lambda\right\}\right|\mr d\lambda,
\end{multline}
where the first equation is justified by a use of Fubini's theorem and $f_k$ is defined by $f_k = \frac{1}{k} \int_0^k \phi^\star(u)\,\mr du > f = \int_0^1 \phi^\star(u)\,\mr du$. \\
The first integral in \eqref{eq:3p6} is obviously equal to $k(f_k)^p = \frac{1}{k^{p-1}}\bigl(\int_0^k\phi^\star\bigr)^p$. We suppose now that $\lambda>f_k$ is fixed. Then there exists $\alpha(\lambda)\in (0,k]$ such that $\frac{1}{\alpha(\lambda)}\int_0^{\alpha(\lambda)}\phi^\star(u)\,\mr du=\lambda$. Note that, without loss of generality, we assume that $\phi^\star(0^+)=+\infty$ (the case $\phi^\star(0^+)<+\infty$ can be handled similarly). As a consequence $\left\{t\in(0,k]: \frac{1}{t}\int_0^t\phi^\star(u)\,\mr du \geq \lambda\right\} = (0,\alpha(\lambda)]$, thus $$\left|\left\{t\in(0,k]: \frac{1}{t}\int_0^k\phi^\star(u)\,\mr du \geq \lambda\right\}\right| = \alpha(\lambda).$$
So the second integral in \eqref{eq:3p6} equals
\[
\int_{\lambda=f_k}^{+\infty}p\lambda^{p-1}\alpha(\lambda)\,\mr d\lambda = \int_{\lambda=f_k}^{+\infty}p\lambda^{p-1}\frac{1}{\lambda}\biggl(\int_0^{\alpha(\lambda)}\phi^\star(u)\biggr)\mr d\lambda,
\]
by the definition of $\alpha(\lambda)$.
The last now integral, equals
\begin{equation} \label{eq:3p7}
\int_{\lambda=f_k}^{+\infty} p\lambda^{p-2} \biggl( \int_{\left\{t\in(0,k]: \frac{1}{t}\int_0^t\phi^\star\geq \lambda\right\}}\! \phi^\star(u)\,\mr du \biggr)\mr d\lambda =
\int_{t=0}^k \frac{p}{p-1}\phi^\star(t) \bigl[\lambda^{p-1}\bigr]_{f_k}^{\frac{1}{t}\int_0^t\phi^\star}\!\mr dt,
\end{equation}
by a use of Fubini's theorem.
As a consequence \eqref{eq:3p6} gives
\begin{multline} \label{eq:3p8}
\int_0^k \left(\frac{1}{t}\int_0^t\phi^\star(u)\,\mr du\right)^p\mr dt = -\frac{1}{p-1}\frac{1}{k^{p-1}}\left(\int_0^k\phi^\star\right)^p + \\
\frac{p}{p-1}\int_0^k\phi^\star(t)\left(\frac{1}{t}\int_0^t\phi^\star\right)^{p-1}\mr dt.
\end{multline}
Then by H\"{o}lder's inequality, applied in the second integral on the right side of \eqref{eq:3p8}, we have that
\begin{multline} \label{eq:3p9}
\int_0^k\left(\frac{1}{t}\int_0^t\phi^\star\right)^p\mr dt \leq
-\frac{1}{p-1}\frac{1}{k^{p-1}}\left(\int_0^k\phi^\star\right)^p + \\
\frac{p}{p-1}\left(\int_0^k[\phi^\star]^p\right)^\frac{1}{p} \left[\int_0^k\left(\frac{1}{t}\int_0^t\phi^\star\right)^p\mr dt\right]^\frac{(p-1)}{p}.
\end{multline}
We set now
\[
J(k) = \int_0^k \left(\frac{1}{t}\int_0^t\phi^\star\right)^p\mr dt,\quad
A(k) = \int_0^k[\phi^\star]^p\quad \text{and}\quad B(k) = \int_0^k\phi^\star.
\]
Then we conclude by \eqref{eq:3p9} that
\begin{align} \label{eq:3p10}
J(k) &\leq -\frac{1}{p-1}\frac{1}{k^{p-1}}[B(k)]^p + \frac{p}{p-1}[A(k)]^\frac{1}{p}[J(k)]^\frac{(p-1)}{p} \Rightarrow \notag \\
\frac{J(k)}{A(k)} &\leq -\frac{1}{p-1}\left(\frac{[B(k)]^p}{k^{p-1}A(k)}\right) + \frac{p}{p-1}\left[\frac{J(k)}{A(k)}\right]^\frac{(p-1)}{p}.
\end{align}
We set now in \eqref{eq:3p10} $\Lambda(k) = \left[\frac{J(k)}{A(k)}\right]^\frac{1}{p}$, thus we get
\begin{align} \label{eq:3p11}
& \Lambda(k)^p \leq -\frac{1}{p-1}\left(\frac{[B(k)]^p}{k^{p-1} [A(k)]}\right) + \frac{p}{p-1}\Lambda(k)^{p-1} \Rightarrow \notag \\
& p[\Lambda(k)]^{p-1} - (p-1)[\Lambda(k)]^p \geq \frac{\bigl(\int_0^k\phi^\star\bigr)^p}{k^{p-1}\int_0^k[\phi^\star]^p} \Rightarrow \notag \\
& H_p(\Lambda(k)) \geq \frac{\bigl(\int_0^k\phi^\star\bigr)^p}{k^{p-1}\int_0^k[\phi^\star]^p} \implies
\Lambda(k) \leq \omega_p\!\left(\frac{\bigl(\int_0^k\phi^\star\bigr)^p}{k^{p-1}\int_0^k[\phi^\star]^p}\right) \Rightarrow \notag \\
& J(k) \leq \int_0^k[\phi^\star]^p\, \omega_p\!\left(\frac{\bigl(\int_0^k\phi^\star\bigr)^p}{k^{p-1}\int_0^k[\phi^\star]^p}\right)^p.
\end{align}
At last by \eqref{eq:3p4}, \eqref{eq:3p5} and \eqref{eq:3p11} we derive the proof of our Lemma.
\end{proof}

We fix now $k\in (0,1]$, and $K\subseteq X$ measurable such that $\mu(K)=k$. Then if $A=A(k)$, $B=B(k)$ are defined as in the proof of Lemma \ref{lem:3p2} we conclude that
\begin{equation} \label{eq:3p12}
\int_K (M_{\mc T}\phi)^p\, d\mu \leq A\,\omega_p\!\left(\frac{B^p}{k^{p-1}A}\right).
\end{equation}
Note now that the $A$, $B$ must satisfy the following conditions
\begin{enumerate}[i)]
\item $B^p \leq k^{p-1}A$, because of H\"{o}lder's inequality for $\phi^\star$ on the interval $(0,k]$.
\item $A\leq F$ and $B\leq f$,
\item $(f-B)^p \leq (1-k)^{p-1}(F-A)$, because of H\"{o}lder's inequality for $\phi^\star$ on the interval $[k,1]$.
\end{enumerate}

\noindent From all the above we conclude the following
\begin{corollary} \label{cor:3p1}
~ \vspace{-6pt}
\[
B_p^\mc T(f,F,k) \leq \sup\left\{ A\,\omega_p\!\left(\frac{B^p}{k^{p-1}A}\right)^p: A,\, B\ \text{satisfy \rnum 1), \rnum 2) and \rnum 3) above} \right\}.
\]
\end{corollary}

\noindent For the next Lemma we fix $0<k<1$ and we consider the function $h_k(B)$, defined by \eqref{eq:2p2}, for $0\leq B\leq f$. Now by the fact that  $\omega_p$ is decreasing on $[0,1]$ and the condition \rnum 3) for $A$, $B$ we immediately conclude the following
\begin{corollary} \label{cor:3p2}
~ \vspace{-6pt}
\begin{multline} \label{eq:3p13}
B_p^\mc T(f,F,k) \leq \sup\left\{ \left(F-\frac{(f-B)^p}{(1-k)^{p-1}}\right)\,\omega_p\!\Biggl(\frac{B^p}{k^{p-1}\left(F-\frac{(f-B)^p}{(1-k)^{p-1}}\right)}\Biggr)^p : \right. \\
\left.
\vphantom{\Biggl(\frac{B^p}{k^{p-1}\left(F-\frac{(f-B)^p}{(1-k)^{p-1}}\right)}\Biggr)^p}
\text{for all}\ B\in [0,f]\ \text{such that}\ h_k(B)\leq F\right\}.
\end{multline}
\end{corollary}

We now prove that we have equality in (3.13). Fix $k\in(0,1]$ and a $B$ which satisfy the conditions stated in Corollary \ref{cor:3p2}. We set $A = F-\frac{(f-B)^p}{(1-k)^{p-1}}$ and we fix also a $\delta\in (0,1)$. \\

We use now Lemma \ref{lem:2p1} to pick a family $\{I_1, I_2, \ldots\}$ of pairwise disjoint elements of $\mc T$ such that $\sum_j \mu(I_j)=k$ and since $\frac{B^p}{k^{p-1}} \leq A$, using the value of $B_p^{\mc T}(f,F,f,1)$ which is evaluated in \cite{14}, for each $j$ we choose a non-negative $\phi_j\in L^p\left(I_j, \frac{1}{\mu(I_j)}\mu\right)$ such that
\begin{gather} \label{eq:3p14}
\int_{I_j} \phi^p\,\mr d\mu = \frac{A}{k}\mu(I_j),\qquad \int_{I_j}\phi\,\mr d\mu = \frac{B}{k}\mu(I_j), \\
\text{and}\ \int_{I_j} \left(\mc M_{\mc T(I_j)}(\phi_j)\right)^p\mr d\mu \; \geq\; \delta\,\frac{A}{k}\,\omega_p\!\left(\frac{B^p}{k^{p-1}A}\right)^p \mu(I_j), \label{eq:3p15}
\end{gather}
where $\mc T(I_j)$ is the subtree of $\mc T$, defined by
\[
\mc T(I_j) = \{I\in\mc T:\, I\subseteq I_j\}.
\]
Next we choose $\psi\in L^p(X\!\setminus\! K, \mu)$ such that $\int_{X\setminus K}\psi^p\,\mr d\mu = F-A > 0$ and $\int_{X\setminus K} \psi\,\mr d\mu = f-B > 0$ which, because of the value of $A$, must be constant and equal to $\frac{f-B}{1-k} = \Big(\frac{(F-A)}{(1-k)}\Big)^\frac{1}{p}$. Here $K$ stands for $K=\cup I_j\subseteq X$.
Then we define $\phi = \psi \chi_{X\setminus K} + \sum_j \phi_j\chi_{I_j}$, and we obviously have
\begin{equation} \label{eq:3p16}
\int_X\phi^p\,\mr d\mu = F\quad \text{and} \quad \int_X\phi\,\mr d\mu = f.
\end{equation}
Additionally we must have by \eqref{eq:3p15} that
\begin{multline} \label{eq:3p17}
\int_K(M_{\mc T}\phi)^p\mr d\mu \geq \delta\, A\, \omega_p\!\left(\frac{B^p}{k^{p-1}A}\right)^p = \\
\delta \left(F-\frac{(f-B)^p}{(1-k)^{p-1}}\right)\omega_p\!\!\left(\frac{B^p}{k^{p-1}\left(F-\frac{(f-B)^p}{(1-k)^{p-1}}\right)}\right)^p.
\end{multline}
Letting $\delta\to1^-$ we obtain equality in \ref{eq:3p13}, thus proving Theorem \ref{thm:1p1}.

\begin{corollary} \label{cor:3p3}
In the statement of Corrolary \ref{cor:3p1} we have equality.
\end{corollary}

\begin{proof}
Immediate, since we have equality on \eqref{eq:3p13}, and the right side of \eqref{eq:3p13} is greater or equal than the right side of the inequality that is stated on  Corrolary \ref{cor:3p1}.
\end{proof}

%
%\eqref and \ref ae not right

%
%
%
%

\section{Construction of the function $g_k$}  \label{sec:4}

We now proceed to prove Theorem 1.2.

\begin{proof}
As it has been proved in Corrolary 3.2, it is true that:
\begin{equation*}
B_p^\mathcal T(f,F,k) = \sup\left\{\mathcal R_k(B): 0\leq B\leq f,\ \text{and}\ h_k(B)\leq F\right\}
\end{equation*}
where $\mathcal R_k(B)$, $h_k(B)$ are defined as in Section \ref{sec:2}.
By Lemma \ref{lem:2p3}, we see that the value $B_0$ satisfies the following equation
\begin{equation} \label{eq:4p1}
\omega_p(Z_0) = \frac{B_0}{k} \frac{1-k}{f-B_0},
\end{equation}
where $Z_0$ is given by
%% \label{eq:4p3}
\begin{equation*}
Z_0 = \frac{B_0^p}{k^{p-1}\left(F-\frac{(f-B_0)^p}{(1-k)^{p-1}}\right)}.
\end{equation*}

We search for a function of the form
\begin{equation}\label{eq:4p2}
g_k(t) = \begin{cases}
A_1\,t^{-1+\frac{1}{a}},& t\in(0,k] \\
c,& t\in(k,1]
\end{cases}
\end{equation}

for some constants $a, c, A_1$, which satisfies the properties
\begin{equation} \label{eq:4p3}
B_p^\mathcal T(f,F,k) = \int_0^k\left(\frac 1 t \int_0^tg_k\right)^p\mathrm dt,
\end{equation}

and

\begin{equation} \label{eq:4p4}
\int_0^1g_k=f,\quad \int_0^1g_k^p=F
\end{equation}

Concerning the first equation in \ref{eq:4p4}, we have
\begin{equation}\label{eq:4p5}
\begin{split}
\int_0^1g_k=f & \Leftrightarrow \int_0^kg_k + \int_k^1g_k = f \Leftrightarrow\\
& \Leftrightarrow \int_0^kg_k+c(1-k)=f.
\end{split}
\end{equation}

We set $c=\frac{f-B_0}{1-k}$, in order to ensure that
\begin{equation} \label{eq:4p6}
\int_0^kg_k = B_0.
\end{equation}

Note that \eqref{eq:4p6} is (in view of \eqref{eq:4p2}) equivalent to
\begin{equation} \label{eq:4p7}
\int_0^kA_1\,t^{-1+\frac 1 a}\mathrm dt = B_0 \Leftrightarrow A_1 = \frac{B_0k^{-1/a}}{a},
\end{equation}
so that we found $A_1$, in terms of $a$. We search now for a value of $a$ such that the second equation in \eqref{eq:4p4} is true. Thus we should have
\begin{gather}
A_1^p \int_0^k t^{-p+\frac{p}{a}}\mathrm dt = F - \frac{(f-B_0)^p}{(1-k)^{p-1}} \Leftrightarrow \nonumber \\
\frac{B_0^p\,k^{-p/a}}{a^p}\frac{1}{1+\frac{p}{a}-p}k^{1-p+p/a} = F-\frac{(f-B_0)^p}{(1-k)^{p-1}} \Leftrightarrow \nonumber \\
\frac{B_0^p}{k^{p-1}}\frac{1}{p\,a^{p-1}-(p{-}1)a^p} = F - \frac{(f-B_0)^p}{(1-k)^{p-1}} \nonumber \Leftrightarrow \\
\frac{B_0^p}{k^{p-1}H_p(a)} = F-\frac{(f-B_0)^p}{(1-k)^{p-1}} \Leftrightarrow H_p(a) = \frac{B_0^p}{k^{p-1}\left(F-\frac{(f-B_0)^p}{(1-k)^{p-1}}\right)} \nonumber \Leftrightarrow \\
H_p(a) = Z_0 \Leftrightarrow a = \omega_p(Z_0) \in \left[1,\frac{p}{p-1}\right] \label{eq:4p8}
\end{gather}
\end{proof}

As a consequence, if $A_1, a$ are given by \eqref{eq:4p7} and \eqref{eq:4p8} respectively, equations \eqref{eq:4p4} are true.
Note now that for every $t\in(0,k]$ we have that
\[
\frac 1 t \int_0^t g_k = a\,g_k(t),\quad \forall t\in(0,k].
\]
Thus
\begin{multline} \label{eq:4p9}
\int_0^k \left(\frac 1 t \int_0^t g_k\right)^p\mathrm dt =a^p \int_0^k g_k^p =\\
= \left(F-\frac{(f-B_0)^p}{(1-k)^{p-1}}\right)\omega_p\!\left(\frac{B_0^p}{k^{p-1}\left(F-\frac{(f-B_0)^p}{(1-k)^{p-1}}\right)}\right)^p.
\end{multline}
%%%%%
By Theorem \ref{thm:1p1} and Lemma 2.3, the right side of \eqref{eq:4p9} equals $B_p^\mathcal T(f,F,k)$. We need only to prove that $g_k$ is continuous on $t_0=k$. It is enough to show that
\begin{multline} \label{eq:4p10}
\frac{f-B_0}{1-k} = A_1k^{-1+\frac 1 a} \Leftrightarrow A_1k^{-1+\frac 1 a} = \left(\frac{B_0}{k} \frac{1-k}{f-B_0}\right)^{-1}\frac{B_0}{k}
\end{multline}

By \eqref{eq:4p1} and \eqref{eq:4p8}, $a=\omega_p(Z_0) = \frac{B_0}{k}\frac{1-k}{f-B_0}$. Thus \eqref{eq:4p10} is equivalent to
$A_1k^{-1+\frac 1 a} = a^{-1}\frac{B_0}{k}$, which is just \eqref{eq:4p7}. Theorem 1.2 is now proved.

\section{The Bellman function of four variables of $\mc M_{\mc T}$ } \label{sec:5}

In this section we evaluate the more general Bellman function of the dyadic maximal operator \eqref{eq:1p6}. More precisely we prove the following

\begin{theorem}\label{thm:5p1}
The following identity holds
\begin{equation} \label{eq:5p1} 
B_p^{\mc T}(f,F,L,k) = \sup \left\{B_p^{\mc T}(f,F,k_1)+L^p(k-k_1):k_1\in(0,k]\right\},
\end{equation}
\begin{proof}
Let $k\in(0,1]$ and $L\geq f$ and consider also $\phi$ as in  \eqref{eq:1p6}. We choose $w\geq L$ such that
$\mu(U_1)\leq k\leq \mu(U_2)$
where $U_1=\left\{max(\mc M_{\mc T}\phi,L)> w\right\}$ and
$U_2=\left\{max(\mc M_{\mc T}\phi,L)\geq w\right\}$. We next consider a measurable set $S$ satisfying $U_1\subseteq S \subseteq U_2$ and
$\mu(S)=k$.

Let now $K$ be an arbitrary measurable subset of $X$ with $\mu(K)=k$. 

We mention the following which is a well known fact from measure theory:

Let $g:(X,\mu)\rightarrow \mathbb{R^+}$ be an integrable function, $w$ non-negative and $D$ be a measurable set such that $\left\{g>w\right\}\subseteq D\subseteq \left\{g \geq w\right\}$. Then for every $K\subseteq X$, with $\mu(K)=\mu(D)$ the inequality
$\int_Kg \mr \; d\mu\leq \int_Dg \mr \; d\mu$ is true.

By using the above fact we conclude that 
\begin{equation} \label{eq:5p2}
\int_K max(\mc M_{\mc T}\phi,L)^p \mr \; d\mu\leq \int_S max(\mc M_{\mc T}\phi,L)^p\mr \; d\mu
\end{equation}

We consider two cases

a) If $w>L$ then $\mc M_{\mc T}\phi>L$ on $S$. Thus \eqref{eq:5p2} gives
\begin{equation} \label{eq:5p3}
\int_K max(\mc M_{\mc T}\phi,L)^p \mr \; d\mu\leq \int_S(\mc M_{\mc T}\phi)^p\mr \; d\mu \leq B_p^{\mc T}(f,F,k)
\end{equation}

b) If $w=L$ then  $\mc M_{\mc T}\phi \leq L$ on $S-U_1$, so that \eqref{eq:5p2} gives
\begin{equation} \label{eq:5p4}
\int_K max(\mc M_{\mc T}\phi,L)^p \mr \; d\mu \leq \int_{U_{1}}(\mc M_{\mc T}\phi)^p\mr \; d\mu +(k-k_1)L^p
\end{equation} 
where we have set $k_1=\mu(U_1)\leq k$. Thus we have proved that
$$B_p^{\mc T}(f,F,L,k) \leq \sup \left\{B_p^{\mc T}(f,F,k_1)+L^p(k-k_1):k_1\in(0,k]\right\}$$

We proceed to the reverse inequality:

Let $k_1\in (0,k]$ and for any $\delta \in (0,1)$ we choose a function   $\phi_{\delta}$ satisfying the conditions in \eqref{eq:1p6}, and a measurable subset $K_{\delta}$ of $X$ with $\mu(K_{\delta})=k_1$ for which 
\begin{equation} \label{eq:5p5}
\int_{K_{\delta}}(\mc M_{\mc T}\phi)^p \mr \; d\mu \geq \delta B_p^{\mc T}(f,F,k_1).  
\end{equation}
We now choose a measurable subset $E$ of $X-K_{\delta}$, such that
$\mu(E)=k-k_1$. We the define $K=K_{\delta} \cup E$. Then $\mu(K)=k$ and
\begin{equation} \label{eq:5p6}
\int_K max(\mc M_{\mc T}\phi,L)^p \mr \; d\mu \geq \int_{K_{\delta}}(\mc M_{\mc T}\phi)^p\mr \; d\mu +(k-k_1)L^p
\end{equation}
so by using \eqref{eq:5p5} and letting also $\delta\rightarrow 1^-$ we conclude the opposite inequality, and thus the proof of Theorem 5.1. 
\end{proof}
\end{theorem}

\noindent Nikolidakis Eleftherios, Assistant Professor, Department of Mathematics, Panepistimioupolis, University of Ioannina, 45110, Greece.
E-mail address: enikolid@uoi.gr


\begin{thebibliography}{99}
\bibitem {1}
    Bekjan, Turdebek N.; Chen, Zeqian; Osekowski, Adam.
    \emph{Noncommutative maximal inequalities associated with convex functions.}
    Trans. Am. Math. Soc. 369 (2017), no. 1, 409--427.

\bibitem {2}
    D. L. Burkholder.
    \emph{Martingales and Fourier analysis in Banach spaces,}
    Probability and analysis (Varenna 1985), 61--108, Lecture Notes in Math., 1206, Springer, Berlin, 1986.

\bibitem {3}
    D. L. Burkholder.
    \emph{Explorations in martingale theory and its applications,}
    {\'E}cole d'{\'E}t{\'e} de Probabilit{\'e}s de Saint-Flour XIX—1989, Springer, Berlin, Heidelberg, 1991. 1-66.

\bibitem{4}
	L. Colzani, J. Perez Lazaro,
	\emph{Eigenfunctions of the Hardy-Littlewood maximal operator},
	Colloq. Math. 118 (2010), no. 2,
	379--389. 	
	
\bibitem{5}
	L. Grafakos, Stephen J. Montgomery-Smith,
	\emph{Best constants for uncentered maximal functions},
	Bull. London Math. Soc. (1997), no. 1,
	60--64.

\bibitem {6}
    Ivanisvili, Paata; Osipov, Nikolay N.; Stolyarov, Dmitriy M.; Vasyunin, Vasily I.; Zatitskiy, Pavel B..
    \emph{Bellman function for extremal problems in BMO,}
    Trans. Am. Math. Soc. 368, (2016), no. 5, 3415--3468.

\bibitem {7}
    Logunov, Alexander A.; Slavin, L.; Stolyarov, D.M.; Vasyunin,
    V.; Zatitskiy, P.B. \emph{Weak integral conditions for BMO},
    Proc. Am. Math. Soc. 143 (2015), no. 7, 2913--2926.

\bibitem {8}A. D. Melas.
    \emph{Sharp general local estimates for dyadic-like maximal operators and related Bellman functions,}
    Adv. in Math. 220 (2009), no. 2, 367--426

\bibitem {9}A. D. Melas.
    \emph{Dyadic-like maximal operators on $L\log L$ functions,}
    J.  Funct. Anal. 257, no. 6, (2009), 1631--1654.
	
\bibitem{10}
	A. D. Melas,
	\emph{The Bellman functions of dyadic-like maximal operators and related inequalities},
	Adv. in Math. 192 (2005), no. 2,
	310--340.
	
\bibitem{11}
	A. D. Melas, E. N. Nikolidakis,
	\emph{Dyadic-like maximal operators on integrable functions and Bellman functions related to Kolmogorov's inequality},
	Trans. Amer. Math. Soc. 362 (2010), no. 3,
	1571--1597.

\bibitem {12}
    F. Nazarov, S. Treil.
    \emph{The hunt for a Bellman function: applications to estimates for singular integral operators and to other classical problems of harmonic analysis,}
    Algebra i Analyz 8 (1996), no. 5,  32--162; translation in St. Petersburg Math. J., 8 (1997), no. 5, 721--824.

\bibitem {13}
    F. Nazarov, S. Treil, A. Volberg.
    \emph{The Bellman functions and two-weight inequalities for Haar multipliers,}
    J. Amer. Math. Soc. 12 (1999), no. 4, 909--928.
	
\bibitem{14}
	E. N. Nikolidakis, A. D. Melas,
          \emph{A sharp integral rearrangement inequality for the dyadic maximal operator and applications},
	Appl. Comput. Harmon. Anal., 38 (2015), no. 2, 242--261.

\bibitem {15}
      G. Wang,
      \emph{ Sharp maximal inequalities for conditionally symmetric martingales and Brownian motion,}
       Proceedings of the American Mathematical Society 112 (1991): 579-586.
	
\end{thebibliography}
\end{document}